\documentclass[a4paper,11pt]{article}
\usepackage{amsmath, amsfonts, amscd, amssymb, amsthm, enumerate, xypic}

\newcommand{\Mat}{\operatorname{M}}

\newcommand{\GL}{\operatorname{GL}}
\newcommand{\Ker}{\operatorname{Ker}}

\newcommand{\im}{\operatorname{Im}}

\def\op{\text{op}}
\newcommand{\rk}{\operatorname{rk}}

\renewcommand{\setminus}{\smallsetminus}

\newcommand{\D}{\mathbb{D}}

\def\F{\mathbb{F}}

\def\calL{\mathcal{L}}

\def\calR{\mathcal{R}}
\def\calS{\mathcal{S}}

\def\calU{\mathcal{U}}
\def\calV{\mathcal{V}}
\def\calW{\mathcal{W}}

\def\calZ{\mathcal{Z}}

\def\lcro{\mathopen{[\![}}
\def\rcro{\mathclose{]\!]}}

\theoremstyle{definition}
\newtheorem{Def}{Definition}
\newtheorem{Not}[Def]{Notation}

\theoremstyle{plain}
\newtheorem{theo}{Theorem}
\newtheorem{prop}[theo]{Proposition}
\newtheorem{cor}[theo]{Corollary}
\newtheorem{lemma}[theo]{Lemma}
\newtheorem{claim}{Claim}

\theoremstyle{plain}

\theoremstyle{remark}

\title{The Flanders theorem over division rings}
\author{Cl\'ement de Seguins Pazzis\footnote{Universit\'e de Versailles Saint-Quentin-en-Yvelines, Laboratoire de Math\'ematiques
de Versailles, 45 avenue des Etats-Unis, 78035 Versailles cedex, France}
\footnote{e-mail address: dsp.prof@gmail.com}}

\begin{document}

\thispagestyle{plain}

\maketitle

\begin{abstract}
Let $\D$ be a division ring and $\F$ be a subfield of the center of $\D$ over which $\D$ has finite dimension $d$.
Let $n,p,r$ be positive integers and $\calV$ be an affine subspace of the $\F$-vector space $\Mat_{n,p}(\D)$
in which every matrix has rank less than or equal to $r$. Using a new method, we prove that $\dim_\F \calV \leq \max(n,p)\,rd$
and we characterize the spaces for which equality holds. This extends a famous theorem of Flanders which was known only for
fields.
\end{abstract}

\vskip 2mm
\noindent
\emph{AMS Classification:} 15A03, 15A30.

\vskip 2mm
\noindent
\emph{Keywords:} Rank, Bounded rank spaces, Flanders's theorem, Dimension, Division ring.


\section{Introduction}

Throughout the text, we fix a division ring $\D$, that is a non-trivial ring in which every non-zero element is invertible.
We let $\F$ be a subfield of the center $\calZ(\D)$ of $\D$ and we assume that $\D$ has finite dimension over $\F$.

Let $n$ and $p$ be non-negative integers. We denote by $\Mat_{n,p}(\D)$ the set of all matrices with $n$ rows, $p$ columns, and entries
in $\D$. It has a natural structure of $\F$-vector space, which we will consider throughout the text.
We denote by $E_{i,j}$ the matrix of $\Mat_{n,p}(\D)$ in which all the entries equal $0$, except the one at the $(i,j)$-spot, which equals $1$.
The right $\D$-vector space $\D^n$ is naturally identified with the space $\Mat_{n,1}(\D)$ of column matrices (with $n$ rows).
We naturally identify $\Mat_{n,p}(\D)$ with the set of all linear maps from the right vector space $\D^p$ to the right vector space $\D^n$.
We have a ring structure on $\Mat_n(\D):=\Mat_{n,n}(\D)$ with unity $I_n$, and its group of units is denoted by $\GL_n(\D)$.

Two matrices $M$ and $N$ of $\Mat_{n,p}(\D)$ are said to be \textbf{equivalent} when there are
invertible matrices $P \in \GL_n(\D)$ and $Q \in \GL_p(\D)$ such that $N=PMQ$ (this means that $M$ and $N$ represent the same
linear map between right vector spaces over $\D$ under a different choice of bases).
This relation is naturally extended to whole subsets of matrices.

The rank of a matrix $M \in \Mat_{n,p}(\D)$ is the rank of the family of its columns in the right $\D$-vector space $\D^n$, and it is known that
it equals the rank of the family of its rows in the left $\D$-vector space $\Mat_{1,p}(\D)$: we denote it by $\rk(M)$.
Two matrices of the same size have the same rank if and only if they are equivalent.

Given a non-negative integer $r$, a \textbf{rank-$\overline{r}$ subset} of $\Mat_{n,p}(\D)$ is a subset in which every matrix
has rank less than or equal to $r$.

Let $s$ and $t$ be non-negative integers. One defines the \textbf{compression space}
$$\calR(s,t):=\biggl\{\begin{bmatrix}
A & C \\
B & [0]_{(n-s) \times (p-t)}
\end{bmatrix} \mid A \in \Mat_{s,t}(\D), \; B \in \Mat_{n-s,t}(\D), \; C \in \Mat_{s,p-t}(\D)\biggr\}.$$
It is obviously an $\F$-linear subspace of $\Mat_{n,p}(\D)$ and a rank-$\overline{s+t}$ subset.
More generally, any space that is equivalent to a space of that form is called a compression space.

A classical theorem of Flanders \cite{Flanders} reads as follows.

\begin{theo}[Flanders's theorem]\label{maintheo}
Let $\F$ be a field, and $n,p,r$ be positive integers such that $n\geq p>r$. Let $\calV$ be a rank-$\overline{r}$ linear subspace of
$\Mat_{n,p}(\F)$.
Then, $\dim \calV \leq nr$, and if equality holds then either
$\calV$ is equivalent to $\calR(0,r)$, or $n=p$ and $\calV$ is equivalent to $\calR(r,0)$.
\end{theo}

The case when $n \leq p$ can be obtained effortlessly by transposing.

\vskip 3mm
Flanders's theorem has a long history dating back to Dieudonn\'e \cite{Dieudonne},
who tackled the case when $n=p$ and $r=n-1$ (that is, subspaces of singular matrices). Dieudonn\'e was motivated by
the study of semi-linear invertibility preservers on square matrices.
Flanders came actually second \cite{Flanders} and, due to his use of determinants, he was only able to prove his results over
fields with more than $r$ elements (he added the restriction that the field should not be of characteristic $2$, but
a close examination of his proof reveals that it is unnecessary).
The extension to general fields was achieved more than two decades later by Meshulam \cite{Meshulam}.
In the meantime, much progress had been made in the classification of rank-$\overline{r}$ subspaces with dimension close to the critical one, over fields with large cardinality (see \cite{AtkLloyd} for square matrices, and \cite{Beasley} for the generalization to rectangular matrices):
the known theorems essentially state that every large enough rank-$\overline{r}$ linear subspace is a subspace of a compression space.

This topic has known a recent revival. First, Flanders's theorem was extended to affine subspaces
over all fields \cite{affpres}, and the result was applied to generalize Atkinson and Lloyd's classification
of large rank-$\overline{r}$ spaces \cite{dSPclass}. Flanders's theorem has also been shown to yield
an explicit description of full-rank-preserving linear maps on matrices without injectivity assumptions \cite{localization,dSPlinpresGL}.

In a recent article, \v{S}emrl proved Flanders's theorem in the case of singular matrices over
division algebras that are finite-dimensional over their center \cite{Semrlcentralsimple},
and he applied the result to classify the linear endomorphisms of a central simple algebra that preserve invertibility.

Our aim here is to give the broadest generalization of Flanders's theorem to date.
It reads as follows:

\begin{theo}\label{Flandersskew}
Let $\D$ be a division ring and $\F$ be a subfield of its center such that $d:=[\D : \F]$ is finite.
Let $n,p,r$ be non-negative integers such that $n \geq p \geq r$.
Let $\calV$ be an $\F$-affine subspace of $\Mat_{n,p}(\D)$, and assume that it is a rank-$\overline{r}$ subset.
Then,
\begin{equation}\label{ineq}
\dim_\F \calV \leq dnr.
\end{equation}
If equality holds in \eqref{ineq}, then:
\begin{enumerate}[(a)]
\item Either $\calV$ is equivalent to $\calR(0,r)$;
\item Or $n=p$ and $\calV$ is equivalent to $\calR(r,0)$;
\item Or $(n,p,r)=(2,2,1)$, $\# \D=2$ and $\calV$ is equivalent to the affine space
$$\calU_2:=\biggl\{\begin{bmatrix}
x & 0 \\
y & x+1
\end{bmatrix} \mid (x,y)\in \D^2\biggr\}.$$
\end{enumerate}
\end{theo}

The case when $n=p$, $r=n-1$, $\F=\calZ(\D)$, $\D$ is infinite and $\calV$ is a linear subspace is Theorem 2.1 of \cite{Semrlcentralsimple}.

As in Flanders's theorem, the case $n \leq p$ can be deduced from our theorem.
Beware however that the transposition does not leave the rank invariant!
If we denote by $\D^{\op}$ the \emph{opposite} division ring\footnote{The ring $\D^{\op}$ has the same underlying abelian group, and its multiplication is defined as $x \underset{\op}{\times} y:=yx$.}, then $A \in \Mat_{n,p}(\D) \mapsto A^T \in \Mat_{p,n}(\D^{\op})$
is an $\F$-linear bijection that is rank preserving and that reverses products. Thus, the case $n \geq p$ over $\D^{\op}$
yields the case $n \leq p$ over $\D$.

Taking $\D$ as a finite field and $\F$ as its prime subfield, we obtain the following corollary on
additive subgroups of matrices:

\begin{cor}
Let $\F$ be a finite field with cardinality $q$. Let $n,p,r$ be positive integers such that $n \geq p \geq r$.
Let $\calV$ be an additive subgroup of $\Mat_{n,p}(\F)$ in which every matrix has rank at most $r$.
Then, $\# \calV \leq q^{nr}$, and if equality holds then either $\calV$ is equivalent to $\calR(0,r)$,
or $n=p$ and $\calV$ is equivalent to $\calR(r,0)$.
\end{cor}

Broadly speaking, the proof of Theorem \ref{Flandersskew} will revive some of Dieudonn\'e's original ideas from \cite{Dieudonne} and
will incorporate some key innovations. The main idea is to work by induction over all parameters $n,p,r$,
with special focus on the rank $1$ matrices in the translation vector space of $\calV$.
In a subsequent article, this new strategy will be used to improve the classification of large rank-$\overline{r}$ spaces
over fields.

The proof of Theorem \ref{Flandersskew} is laid out as follows:
Section \ref{lemma} consists of a collection of three basic lemmas.
The inductive proof of Theorem \ref{Flandersskew} is then performed in the final section.

\section{Basic results}\label{lemma}

\subsection{Extraction lemma}

\begin{lemma}[Extraction lemma]\label{extractionlemma}
Let $M=\begin{bmatrix}
A & C \\
B & d
\end{bmatrix}$ be a matrix of $\Mat_{n,p}(\D)$, with $A \in \Mat_{n-1,p-1}(\D)$.
Assume that $\rk(M) \leq r$ and $\rk(M+E_{n,p}) \leq r$.
Then, $\rk A \leq r-1$.
\end{lemma}

\begin{proof}
Assume on the contrary that $\rk A \geq r$.
Without loss of generality, we can assume that
$$A=\begin{bmatrix}
I_r & [?]_{r \times (p-1-r)} \\
[?]_{(n-1-r) \times r} & [?]_{(n-1-r)\times (p-1-r)}
\end{bmatrix}.$$
We write $C=\begin{bmatrix}
C_1 \\
[?]_{(n-1-r)}
\end{bmatrix}$ and $B=\begin{bmatrix}
B_1 & [?]_{1 \times (p-r-1)}
\end{bmatrix}$ with $C_1 \in \Mat_{r,1}(\D)$ and $B_1 \in \Mat_{1,r}(\D)$.
Then, by extracting sub-matrices, we find that for all $\delta \in \{0,1\}$, the matrix
$$H_\delta:=\begin{bmatrix}
I_r & C_1 \\
B_1 & d+\delta
\end{bmatrix}$$
has rank less than or equal to $r$.
Multiplying it on the left with the invertible matrix $\begin{bmatrix}
I_r & [0]_{r \times 1} \\
-B_1 & 1
\end{bmatrix}$, we deduce that
$$\forall \delta \in \{0,1\}, \quad \rk \begin{bmatrix}
I_r & C_1 \\
[0]_{1 \times r} & d+\delta-B_1C_1
\end{bmatrix} \leq r.$$
This would yield
$$\forall \delta \in \{0,1\}, \quad d+\delta-B_1C_1=0,$$
which is absurd. Thus, $\rk A <r$, as claimed.
\end{proof}

\subsection{Range-compatible homomorphisms on $\Mat_{n,p}(\D)$}\label{rangecompatiblesection}

\begin{Def}
Let $U$ and $V$ be right vector spaces over $\D$, and $\calS$ be a subset of $\calL(U,V)$, the set of all linear maps from $U$ to $V$.
A map $F : \calS \rightarrow V$ is called \textbf{range-compatible} whenever
$$\forall s \in \calS, \; F(s) \in \im s.$$
\end{Def}

The concept of range-compatibility was introduced and studied in \cite{dSPRC1}.
Here, we shall need the following basic result:

\begin{prop}\label{RClemma}
Let $n,p$ be non-negative integers with $n \geq 2$, and
$F : \Mat_{n,p}(\D) \rightarrow \D^n$ be a range-compatible (group) homomorphism.
Then, $F : M \mapsto MX$ for some $X \in \D^p$.
\end{prop}

Of course here $\Mat_{n,p}(\D)$ is naturally identified with the group of all linear mappings from $\D^p$ to $\D^n$.

\begin{proof}
We start with the case when $p=1$. Then, $F$ is simply an endomorphism of $\D^n$ such that
$F(X)$ is (right-)collinear to $X$ for all $X \in \D^n$. Then, it is well-known that $F$ is a right-multiplication map:
we recall the proof for the sake of completeness.
For every non-zero vector $X \in \D^n$, we have a scalar $\lambda_X \in \D$ such that $F(X)=X\,\lambda_X$.
Given non-collinear vectors $X$ and $Y$ in $\D^n$, we have
$$X\,\lambda_{X+Y}+Y\,\lambda_{X+Y}=(X+Y)\,\lambda_{X+Y}=F(X+Y)=F(X)+F(Y)=X\,\lambda_X+Y\,\lambda_Y$$
and hence $\lambda_X=\lambda_{X+Y}=\lambda_Y$. Given collinear non-zero vectors $X$ and $Y$ of $\D^n$, we can find
a vector $Z$ in $\D^n \setminus (X \D)$ (since $n \geq 2$), and it follows from the first step that $\lambda_X=\lambda_Z=\lambda_Y$.
Thus, we have a scalar $\lambda \in \D$ such that $F(X)=X\lambda$ for all $X \in \D^n \setminus \{0\}$, which holds
also for $X=0$.

Now, we extend the result. As $F$ is additive there are group endomorphisms $F_1,\dots,F_p$ of $\D^n$ such that
$$F : \begin{bmatrix}
C_1 & \cdots & C_p
\end{bmatrix} \mapsto \sum_{k=1}^p F_k(C_k).$$
By applying the range-compatibility assumption to matrices with only one non-zero column, we see that each map $F_k$ is range-compatible.
This yields scalars $\lambda_1,\dots,\lambda_p$ in $\D$ such that
$$F : \begin{bmatrix}
C_1 & \cdots & C_p
\end{bmatrix} \mapsto \sum_{k=1}^p C_k\,\lambda_k.$$
Thus, with $X:=\begin{bmatrix}
\lambda_1 \\
\vdots \\
\lambda_p
\end{bmatrix}$ we find that $F : M \mapsto MX$, as claimed.
\end{proof}

\subsection{On the rank $1$ matrices in the translation vector space of a rank-$\overline{r}$ affine subspace}\label{rank1}

\begin{Not}
Let $S$ be a subset of $\Mat_{n,p}(\D)$ and
$H$ be a linear hyperplane of the $\D$-vector space $\D^p$. We define
$$S_H:=\{M \in S : \; H \subset \Ker M\}.$$
\end{Not}

Note that $S_H$ is an $\F$-linear subspace of $S$ whenever $S$ is an $\F$-linear subspace of $\Mat_{n,p}(\D)$.

Let $\calS$ be an $\F$-affine rank-$\overline{k}$ space, with translation vector space $S$.
In our proof of Flanders's theorem, we shall need to find a hyperplane $H$
such that the dimension of $S_H$ is small.
This will be obtained through the following key lemma.

\begin{lemma}\label{keylemma}
Let $\calV$ be an $\F$-affine rank-$\overline{r}$ subspace of $\Mat_{n,p}(\D)$, with $p>r>0$.
Denote by $V$ its translation vector space. Assume that
$\dim_\F V_H \geq dr$ for every linear hyperplane $H$ of the $\D$-vector space $\D^p$.
Then, $\calV$ is equivalent to $\calR(r,0)$.
\end{lemma}

\begin{proof}
Denote by $s$ the maximal rank among the elements of $\calV$.
Let us consider a matrix $A$ of $\calV$ with rank $s$.
Given a linear hyperplane $H$ of $\D^p$ that does not include $\Ker A$,
let us set
$$T_H:=\sum_{M \in V_H} \im M.$$
We claim that $T_H=\im A$.
To support this, we lose no generality in assuming that
$$A=J_s:=\begin{bmatrix}
I_s & 0 \\
0 & 0
\end{bmatrix} \quad \text{and $H=\D^{p-1} \times \{0\}$.}$$
The elements of $V_H$ are the matrices of $V$ whose columns are all zero with the possible exception of the last one.
For an arbitrary element $N$ of $V_H$, we must have $\rk(A+N) \leq s$ as $A+N$ belongs to $\calV$.
This shows that the last $n-s$ rows of $N$ must equal zero.
It follows that $\dim_\F V_H \leq ds$, and our assumptions show that we must have $s=r$
and $\dim_\F V_H=dr$. In turn, this shows that $V_H$ is the set of all matrices of $\Mat_{n,p}(\D)$ with non-zero entries
only in the last column and in the first $r$ rows, and it is then obvious that $T_H=\D^r \times \{0\}= \im A$.

Now, let us get back to the general case. Without loss of generality, we can assume that $\calV$ contains $A=J_r$.
Then, taking $H=\D^{p-1} \times \{0\}$ shows that $V$ contains $E_{1,p},\dots,E_{r,p}$.
Given $(i,j) \in \lcro 1,r\rcro \times \lcro 1,p-1\rcro$, taking $H=\{(x_1,\dots,x_p) \in \D^p : \; x_j=x_p\}$ shows that
$V$ contains $E_{i,j}-E_{i,p}$, and as the $\F$-vector space $V$ also contains $E_{i,p}$
we conclude that it contains $E_{i,j}$.
It follows that $\calR(r,0) \subset V$. As $J_r \in \calV$, we deduce that $0 \in \calV$, and hence $\calV=V$.

Finally, assume that some matrix $N$ of $\calV$ has a non-zero row among the last $n-r$ ones:
then, we know from $\calR(r,0) \subset V$ that $\calV$ contains every matrix of $\Mat_{n,p}(\D)$
with the same last $n-r$ rows as $N$; at least one such matrix has rank greater than $r$, obviously.
Thus, $\calV \subset \calR(r,0)$ and we conclude that $\calV=\calR(r,0)$.
\end{proof}

\section{The proof of Theorem \ref{Flandersskew}}

We are now ready to prove Theorem \ref{Flandersskew}.
We shall do this by induction over $n,p,r$.
The case $r=0$ is obvious, and so is
the case $r=p$.

Assume now that $1\leq r<p$.
Let $\calV$ be an $\F$-affine subspace of $\Mat_{n,p}(\D)$ in which every matrix has rank at most $r$.
Denote by $V$ its translation vector space.

If $\calV$ is equivalent to $\calR(r,0)$, then it has dimension $dpr$ over $\F$, which is less than or equal to $dnr$.
Moreover, if equality occurs then $n=p$ and we have conclusion (b) in Theorem \ref{Flandersskew}.
In the rest of the proof, we assume that $\calV$ is inequivalent to $\calR(r,0)$.

Thus, Lemma \ref{keylemma} yields a linear hyperplane $H$ of $\D^p$ such that
$\dim_\F V_H < dr$. Without loss of generality, we can assume that $H=\D^{p-1} \times \{0\}$.
From there, we split the discussion into two subcases, whether $V_H$ contains a non-zero matrix or not.

\subsection{Case 1: $V_H \neq \{0\}$.}

Without further loss of generality, we can assume that $V_H$ contains $E_{n,p}$.
Let us split every matrix $M$ of $\calV$ into
$$M=\begin{bmatrix}
K(M) & [?]_{(n-1) \times 1} \\
[?]_{1 \times (p-1)} & ?
\end{bmatrix} \quad \text{with $K(M) \in \Mat_{n-1,p-1}(\D)$.}$$
Then, by the extraction lemma, we see that
$K(\calV)$ is an $\F$-affine rank-$\overline{r-1}$ subspace of $\Mat_{n-1,p-1}(\D)$.
By induction,
$$\dim_\F K(\calV) \leq d(n-1)(r-1).$$
On the other hand, by the rank theorem
$$\dim_\F \calV \leq \dim_\F K(\calV)+d(p-1)+\dim_\F V_H.$$
Hence,
$$\dim_\F \calV < d(n-1)(r-1)+d(p-1)+dr=d(nr+p-n) \leq dnr.$$
Thus, in this situation we have proved inequality \eqref{ineq}, and equality cannot occur.

\subsection{Case 2: $V_H=\{0\}$.}

Here, we split every matrix $M$ of $\calV$ into
$$M=\begin{bmatrix}
A(M) & [?]_{n \times 1}
\end{bmatrix} \quad \text{with $A(M) \in \Mat_{n,p-1}(\D)$.}$$
Then, $A(\calV)$ is an $\F$-affine rank-$\overline{r}$ subspace of $\Mat_{n,p-1}(\D)$, and as $V_H=\{0\}$ we have
$$\dim_\F A(\calV)=\dim_\F \calV.$$
By induction we have
$$\dim_\F A(\calV) \leq dnr$$
and hence
$$\dim_\F \calV \leq dnr.$$
Assume now that $\dim_\F \calV=dnr$, so that $\dim_\F A(\calV)=dnr$.
As $n>p-1$, we know by induction that $A(\calV)$ is equivalent to $\calR(0,r)$
(cases (b) and (c) in Theorem \ref{Flandersskew} are barred).
Without loss of generality, we may now assume that $A(\calV)=\calR(0,r)$.
Then, as $V_H=\{0\}$ we have an $\F$-affine map $F : \Mat_{n,r}(\D) \rightarrow \D^n$
such that
$$\calV=\biggl\{\begin{bmatrix}
N & [0]_{n \times (p-1-r)} & F(N)
\end{bmatrix}\mid N \in \Mat_{n,r}(\D)\biggr\}.$$

\begin{claim}
The map $F$ is range-compatible unless $\# \D=2$ and $(n,p,r)=(2,2,1)$.
\end{claim}

\begin{proof}
Throughout the proof, we assume that we are not in the situation where $\# \D=2$, $n=p=2$ and $r=1$.

Set $G_1:=\D^{n-1} \times \{0\}$. Let us prove that $F(N) \in G_1$ for all $N \in \Mat_{n,r}(\D)$ such that
$\im N \subset G_1$. We have an $\F$-affine mapping $\gamma : \Mat_{n-1,r}(\D)\rightarrow \D$ such that
$$\forall R \in \Mat_{n-1,r}(\D), \quad
F\left(\begin{bmatrix}
R \\
[0]_{1 \times r}
\end{bmatrix}\right)=\begin{bmatrix}
[?]_{(n-1) \times 1} \\
\gamma(R)
\end{bmatrix}$$
and we wish to show that $\gamma$ vanishes everywhere on $\Mat_{n-1,r}(\D)$.
Assume that this is not true and choose a non-zero scalar $a \in \D \setminus \{0\}$ in the range of $\gamma$.
Then $\calW:=\gamma^{-1} \{a\}$ is an $\F$-affine subspace of $\Mat_{n-1,r}(\D)$ with codimension at most $d$.
Moreover, as $\calV$ is a rank-$\overline{r}$ space it is obvious that every matrix in $\calW$ has rank at most $r-1$.
Thus, by induction we know that $\dim_\F \calW \leq d(n-1)(r-1)$.
This leads to $(n-1)(r-1) \geq (n-1)r-1$, and hence $n \leq 2$.
Assume that $n=2$, so that $p=2$ and $r=1$. Then, $\# \D \neq 2$.
Moreover, we must have $\dim_\F \calW \leq d(n-1)(r-1)=0$ and hence $\gamma$ is one-to-one. As $\D$ has more than $2$ elements
this yields a non-zero scalar $b \in \D$ such that $\gamma(b) \neq 0$, yielding a rank $2$ matrix in $\calV$.
Therefore, in any case we have found a contradiction.
Thus, $\gamma$ equals $0$.

We conclude that $F(N) \in G_1$ for all $N \in \Mat_{n,r}(\D)$ such that
$\im N \subset G_1$. Using row operations, we generalize this as follows: for every linear hyperplane $G$ of the right $\D$-vector space $\D^n$
and every matrix $N \in \Mat_{n,r}(\D)$, we have the implication
$$\im N \subset G \Rightarrow F(N) \in G.$$

Let then $N \in \Mat_{n,r}(\D)$. We can find linear hyperplanes $G_1,\dots,G_k$ of the right $\D$-vector space $\D^n$ such that
$\im N=\underset{i=1}{\overset{k}{\bigcap}} G_i$, and hence
$$F(N) \in \underset{i=1}{\overset{k}{\bigcap}} G_i=\im N.$$
Thus, $F$ is range-compatible.
\end{proof}

Now we can conclude. Assume first that we are not in the special situation where $(n,p,r,\# \D)=(2,2,1,2)$.
Then, we have just seen that $F$ is range-compatible. In particular this shows that $F(0)=0$, and as
$F$ is $\F$-affine we obtain that $F$ is a group homomorphism.
Proposition \ref{RClemma} yields that $F : N \mapsto NX$ for some $X \in \D^r$.
Setting
$$P:=\begin{bmatrix}
I_r & [0]_{r \times (p-1-r)} & -X \\
[0]_{(p-1-r) \times r} & I_{p-1-r} & [0]_{(p-1-r) \times 1} \\
[0]_{1 \times r} & [0]_{1 \times (p-1-r)} & 1
\end{bmatrix},$$
we see that $P$ is invertible and that $\calV\, P=\calR(0,r)$, which completes the proof in that case.

Assume finally that $(n,p,r,\# \D)=(2,2,1,2)$. Then,
$$F : \begin{bmatrix}
x \\
y
\end{bmatrix} \mapsto \begin{bmatrix}
\alpha x+\beta y+\gamma \\
\delta x+\epsilon y+\eta
\end{bmatrix}$$
for fixed scalars $\alpha,\beta,\gamma,\delta,\epsilon,\eta$.
As every matrix in $\calV$ is singular, we deduce that, for all $(x,y)\in \D^2$,
$$\begin{vmatrix}
x & \alpha x+\beta y+\gamma \\
y & \delta x+\epsilon y+\eta
\end{vmatrix}=0,$$
and hence
$$(\epsilon+\alpha)xy+(\delta+\eta)x+(\beta+\gamma)y=0.$$
Thus, $\epsilon=\alpha$, $\delta=\eta$ and $\beta=\gamma$.
Performing the column operation $C_2 \leftarrow C_2-\alpha C_1$ on $\calV$, we see that no
generality is lost in assuming that $\alpha=0$.
Then,
$$\calV=\biggl\{\begin{bmatrix}
x & \beta(y+1) \\
y & \delta(x+1)
\end{bmatrix}\mid (x,y)\in \D^2\biggr\}$$
and there are four options to consider:
\begin{itemize}
\item If $\beta=\delta=0$, then $\calV=\calR(0,1)$;
\item If $\beta=0$ and $\delta=1$, then $\calV=\calU_2$;
\item If $\beta=1$ and $\delta=0$, then the row swap $L_1 \leftrightarrow L_2$ takes $\calV$ to $\calU_2$;
\item Finally, if $\beta=\delta=1$, then the column operation $C_2 \leftarrow C_2+C_1$ followed by the row
operation $L_1 \leftarrow L_1+L_2$ takes $\calV$ to $\calU_2$.
\end{itemize}
This completes the proof of Theorem \ref{Flandersskew}.

\end{document}